\documentclass[11pt,a4paper]{article}
\usepackage[utf8]{inputenc}
\usepackage{amsmath}
\usepackage{amsfonts}
\usepackage{amssymb}
\usepackage{graphicx}
\usepackage{amsthm}
\usepackage{setspace}
\usepackage[a4paper, margin=1.25in]{geometry}

\newtheorem{theorem}{Theorem}

\newtheorem{lemma}[theorem]{Lemma}
\newtheorem{cor}[theorem]{Corollary}
\theoremstyle{definition}
\newtheorem{constr}[theorem]{Construction}

\def\sref#1{Section~$\ref{#1}$}
\def\lref#1{Lemma~$\ref{#1}$}
\def\tref#1{Theorem~$\ref{#1}$}
\def\Tref#1{Table~$\ref{#1}$}
\def\fref#1{Figure~$\ref{#1}$}
\def\cref#1{Corollary~$\ref{#1}$}

\def\conref#1{Construction~$\ref{#1}$}

\newcommand{\sym}{\mathcal{S}}

\renewcommand{\geq}{\geqslant}
\renewcommand{\leq}{\leqslant}

\title{Excess Coverage Arrays and Levenshtein's Conjecture}
\author{Amber E. Gentle, Daniel Horsley and Ian M. Wanless\\
\small School of Mathematics\\[-0.5ex]
\small Monash University\\[-0.5ex]
\small Vic 3800, Australia\\
\small\tt \{amber.gentle, daniel.horsley, ian.wanless\}@monash.edu}
\date{}

\begin{document}
\maketitle

\begin{abstract}

A sequence covering array, denoted \textsf{SCA}$(N;t,v)$, is a set of $N$ permutations of $\{0, \dots, v-1 \}$ such that each sequence of $t$ distinct elements of $\{0, \dots, v-1\}$ reads left to right in at least one permutation. The minimum number of permutations such a sequence covering array can have is $t!$ and Levenshtein conjectured that if a sequence covering array with $t!$ permutations exists, then $v \in \{t,t+1\}$. In this paper, we prove that if an \textsf{SCA}$(7!;7,v)$ exists, then $v \leq 9$. We do this by analysing connections between sequence covering arrays and a special kind of covering array called an excess coverage array. A strength 2 excess coverage array, denoted \textsf{CA}$_{X}(N;2,k,v)$, is an $N \times k$ array with entries from $\{0, \dots, v - 1\}$ such that every ordered pair of distinct symbols appear at least once in each pair of columns and all other pairs appear at least twice. We demonstrate computationally that there is a unique \textsf{CA}$_{X}(42;2,5,6)$, and we prove that this array alone does not satisfy necessary conditions we establish for the existence of an \textsf{SCA}$(7!;7,10)$. Furthermore, we find the maximum possible number of columns for strength 2 excess coverage arrays with symbol sets of sizes between 2 and 6, and more broadly investigate binary excess coverage arrays.

\end{abstract}

\section*{Keywords}
Sequence covering array, covering array, orthogonal array, exact covering

\section*{Mathematics Subject Classification}
05B30, 05B40, 05B15

\section{Introduction}\label{intro3}

Let $v$ and $t$ be positive integers with $v \geqslant t$. We let $[v] = \{ 0, \dots, v-1 \}$. We then let $\mathcal{S}_{v}$ be the set of permutations of $[v]$. Unless stated otherwise, permutations are assumed to be written in one-line notation with $\pi \in \mathcal{S}_{v}$ being denoted by $\pi(0) \pi(1) \cdots \pi(v-1)$. We let $\mathcal{S}_{v, t}$ be the set of ordered sequences of $t$ distinct elements of $[v]$. Unless stated otherwise, we write a sequence $s \in \mathcal{S}_{v,t}$ as $(s_{0},s_{1}, \dots, s_{t-1})$. A permutation $\pi \in \mathcal{S}_{v}$ \textit{covers} $s \in \mathcal{S}_{v,t}$ if $\pi^{-1}(s_{i}) < \pi^{-1}(s_{i+1})$ for $0 \leqslant i \leqslant t-2$.

A \textit{sequence covering array} with parameters $N$, $v$ and $t$, denoted by \textsf{SCA}$(N;t,v)$, is a set $X$ of $N$ permutations in $\mathcal{S}_{v}$ such that for each sequence $s \in \mathcal{S}_{v,t}$, there exists $\pi \in X$ such that $\pi$ covers $s$. We refer to $N$ as the \textit{size}, $v$ as the \textit{order}, and $t$ as the \textit{strength} of an \textsf{SCA}$(N;t,v)$. Sequence covering arrays were first studied by Spencer~\cite{Spen71} in the 1970s and have gained recent attention for their applications in software testing~\cite{BTI12, Brain12, Kuhn12, MC15, TJ22}. As such, the goal of research is often to find sequence covering arrays with as few permutations as possible. For given $v$ and $t$, we let \textsf{SCAN}$(t,v)$ be the smallest $N$ such that there exists an \textsf{SCA}$(N;t,v)$.

Given a $t$-subset $T \subseteq [v]$, there are $t!$ ways of arranging the elements of $T$. In order to cover all of these arrangements, a sequence covering array requires at least $t!$ permutations. Moreover, in an \textsf{SCA}$(t!;t,v)$, every sequence must be covered exactly once. The problem of determining when \textsf{SCAN}$(t,v) = t!$ has received particular attention in part because of its connections to coding theory~\cite{Lev91}. Any permutation in $\sym_{v}$ is trivially an \textsf{SCA}$(1; 1,v)$. Any permutation in $\sym_{v}$ and its reverse form an \textsf{SCA}$(2; 2, v)$, and the set $\sym_{t}$ forms an \textsf{SCA}$(t!; t, t)$. Hence we need only consider this problem for $3 \leq t \leq v - 1$. Levenshtein proved by construction that an \textsf{SCA}$(t!; t, t+1)$ exists for $t \geq 2$~\cite{Lev91} and conjectured that for $t \geq 3$, an \textsf{SCA}$(t!;t,v)$ exists if and only if $v \in \{t, t+1 \}$~\cite{Levconj}. This conjecture was later revised by Mathon and Tran van Trung~\cite{Math99} as they had found an \textsf{SCA}$(4!; 4, 6)$. However, this is the only known counter-example to Levenshtein's conjecture and the conjecture has otherwise been confirmed for $t \leq 6$~\cite{Math99}. Moreover, it is known that an \textsf{SCA}$(4!; 4, 7)$ does not exist~\cite{Kle04}. In this paper, we prove the following theorem.







\begin{theorem}
If an \textup{\textsf{SCA}}$(7!;7,v)$ exists, then $v \in \{7,8,9\}$. \label{t:sca79}
\end{theorem}

This reduces the best known upper bound on order of an \textsf{SCA}$(7!;7,v)$ from 13 to 9. It is unknown whether an \textsf{SCA}$(7!;7,9)$ exists.

We prove \tref{t:sca79} by analysing a special kind of covering array called an excess coverage array introduced by Chee et al.~\cite{Chee13}. Let $\textsf{C}$ be an $N \times k$ array where each entry in $\textsf{C}$ is a symbol from the alphabet $[v]$. A \textit{$t$-way interaction} is a set of $t$ pairs $\{ (c_i, \nu_i) : 0 \leq i \leq t-1 \}$ where each $c_i$ is a column of $\textsf{C}$ such that $c_i \neq c_j$ for $i \neq j$, and each $\nu_i$ is an element of $[v]$. The row $\rho$ of $C$ \textit{covers} the interaction $\{ (c_i, \nu_i) : 0 \leq i \leq t-1 \}$ if the entry of $\textsf{C}$ in row $\rho$ and column $c_i$ is $\nu_i$ for $0 \leq i \leq t-1$. The array $\textsf{C}$ is a \textit{covering array of strength $t$}, denoted by $\textsf{CA}(N; t,k,v)$, if for each $t$-way interaction $T$, there is some row of $\textsf{C}$ that covers $T$.

For an interaction $T = \{ (c_i, \nu_i) : 0 \leq i \leq t-1 \}$, and for $\sigma \in [v]$, let $\tau_{\sigma}(T) =  \{ c_i : \nu_i = \sigma \} $. Then, let $\mu(T) = \prod_{\sigma = 0}^{v - 1} \vert \tau_{\sigma}(T) \vert !$. Then $\textsf{C}$ is an \textit{excess coverage array}, denoted by $\textsf{CA}_{X}(N; t,k,v)$, if each $t$-way interaction $T$ is covered by at least $\mu(T)$ different rows of $C$. We consider a \textsf{CA}$_{X}(N;t,k,v)$ as a multiset of rows where each row is a vector in $[v]^{k}$. In particular, reordering the rows of a \textsf{CA}$_{X}(N;t,k,v)$ does not generate a different array.

Excess coverage arrays were introduced by Chee et al.~\cite{Chee13} as a means of drawing connections between sequence covering arrays and covering arrays. They describe how to construct excess coverage arrays from sequence covering arrays and, in particular, show that the existence of an \textsf{SCA}$(t!;t,v)$ implies the existence of a \textsf{CA}$_{X}(t(t-1); 2, v - t + 2, t - 1)$. 


We perform computations to find the maximum number of columns $k$ for which a \textsf{CA}$_{X}(v(v+1);2,k,v)$ exists. We also count the number of isomorphism classes of \textsf{CA}$_{X}(v(v+1);2,k,v)$ for each possible pair $(k,v)$ for $2 \leq v \leq 6$. Two excess coverage arrays are isomorphic if one can be obtained from the other by permuting the symbols and permuting the columns of the array. A row of an excess coverage array is \textit{constant} if it contains only one symbol. Our computations prove the following result. 

\begin{theorem}
\label{unique} Up to isomorphism, there is a unique \textsf{CA}$_{X}(42; 2, 5, 6)$. It contains no constant rows.
\end{theorem}

This unique \textsf{CA}$_{X}(42;2,5,6)$ is shown in \fref{fig:design}. We will extend the work of Chee et al.~\cite{Chee13} to show that the existence of an \textsf{SCA}$(t!;t,v)$ implies that every possible multiset of $k$ elements of $[v]$ must appear in a row of some \textsf{CA}$_{X}(t(t-1); 2, v - t + 2, t - 1)$. Thus, the fact that no \textsf{CA}$_{X}(42;2,5,6)$ contains a constant row implies \tref{t:sca79}. 

In an \textsf{SCA}$(t!;t,v)$, every sequence of length $t$ is covered exactly once. A \textit{perfect sequence covering array} with order $v$, strength $t$ and multiplicity $\lambda$, denoted by \textsf{PSCA}$(v,t,\lambda)$, is a multiset of permutations in $\sym_{v}$ such that every sequence in $\sym_{v,t}$ is covered by exactly $\lambda$ permutations. Then, an \textsf{SCA}$(t!;t,v)$ is also a \textsf{PSCA}$(v,t,1)$. Perfect sequence covering arrays were introduced by Yuster~\cite{Yus19} in 2020 as a design theoretical variant of sequence covering arrays. For more on perfect sequence covering arrays, see~\cite{Ge23, GeWa22, Iur22, NJL23, Yus19}.

In \sref{paperexcess}, we consider excess coverage arrays as introduced by Chee et al.~\cite{Chee13}. We perform our own general analysis of these objects as they pertain to sequence covering arrays, including deriving necessary conditions for the existence of sequence covering arrays. In \sref{ca2}, we focus specifically on excess coverage arrays of strength 2 and describe computations that prove \tref{unique}. In \sref{oa2}, we compare excess coverage arrays of strength 2 and orthogonal arrays of strength 2. In \sref{binary}, we discuss excess coverage arrays whose symbol set contains 2 elements. Finally, we conclude in \sref{concl} with some open problems.

\section{Covering arrays with excess coverage}\label{paperexcess}

In this section, we discuss the connections between sequence covering arrays and excess coverage arrays established by Chee et al.~\cite{Chee13}. We extend these results in \lref{t:icax}. We then use this lemma to prove \lref{t:mult}, which in turn allows us to deduce \tref{t:sca79} from \tref{unique}. Let \textsf{CAN}$(t,k,v)$ be the smallest $N$ for which a \textsf{CA}$(N;t,k,v)$ exists and let \textsf{CAN}$_{X}(t,k,v)$ be the smallest $N$ for which a \textsf{CA}$_{X}(N;t,k,v)$ exists. 

\begin{constr}\label{construction}
Let $X$ be an \textsf{SCA}$(N;t,v)$ and let $A$ be a set of $a$ symbols of $[v]$ with $0 < a < t$. Let $u$ be some permutation of the symbols of $A$, let $X_{u}$ be the set of permutations in $X$ that cover $u$, and let $n = \vert X_{u} \vert$. Now, let $\textsf{C}$ be an $n \times (v - a)$ array with rows indexed by the permutations in $X_{u}$ and columns indexed by the symbols in $[v] \setminus A$. The entry of row $\rho$ and column $\nu$ of $\textsf{C}$ is the number of symbols of $A$ that precede $\nu$ in the permutation $\rho$.
\end{constr}

Note that \conref{construction} can be performed for any choice of $A$ and then for any choice of $u$. Hence, we can build $\binom{v}{a}a! = v!/(v - a)!$ arrays in this way. We let $\mathcal{C}_{a}(X)$ be the multiset of all the arrays that can be built through the removal of $a$ symbols from $X$ via \conref{construction}. Chee et al.~\cite{Chee13} prove that each array $\textsf{C} \in \mathcal{C}_{a}(X)$ is a \textsf{CA}$_{X}( \vert \textsf{C} \vert; 2, v - t + 2, t - 1)$. The proof is short, and we provide it here for completeness with a small extension in the case where $X$ is an \textsf{SCA}$(t!;t,v)$

\begin{lemma}\label{t:icax}
Let $X$ be an \textup{\textsf{SCA}}$(N;t,v)$  and let $0 < a < t$. Then every array $\textup{\textsf{C}} \in \mathcal{C}_{a}(X)$ is a \textup{\textsf{CA}}$_{X}( \vert \textup{\textsf{C}} \vert; 2, v - t + 2, t - 1)$. Moreover, if $X$ is an \textup{\textsf{SCA}}$(t!; t,v)$, then each $(t-a)$-way interaction $T$ must be covered by exactly $\mu(T)$ rows of $\textup{\textsf{C}}$.
\end{lemma}

\begin{proof}
Let $X$ be an \textsf{SCA}$(N;t,v)$ and let $\textsf{C} \in \mathcal{C}_{a}(X)$. Then the rows of $\textsf{C}$ correspond to the permutations in $X$ that cover some permutation $(x_{0}, \dots, x_{a-1})$ of the elements of some symbol set $A$. Let $D = \{y_{0}, \dots, y_{t - a - 1} \}$ be a subset of $[v]$ that is disjoint from $A$ and let $T = \{ ( y_{i}, \nu_{i} ) : 0 \leq i \leq t - a - 1 \}$ be a $(t-a)$-way interaction where each $\nu_{i}$ is an element of $[a+1]$. Consider the permutations of $A \cup D$ that cover $(x_{0}, \dots, x_{a-1})$ and have the elements of $\tau_{0}(T)$ preceding $x_{0}$, the elements of $\tau_{a}(T)$ succeeding $x_{a-1}$ and the elements of $\tau_{i}(T)$ lying between $x_{i-1}$ and $x_{i}$ for $1 \leq i \leq a-1$. There are $\mu(T)$ such permutations. Each of these corresponds to a sequence in $\mathcal{S}_{v,t}$ which must be covered by at least one permutation in $X$. Moreover, a row of $\textsf{C}$ covers $T$ if and only if the corresponding permutation in $X$ covers one of these sequences. Therefore, there are at least $\mu(T)$ rows of $\textsf{C}$ that cover $T$.

If $X$ is an \textsf{SCA}$(t!;t,v)$, then each sequence of length $t$ is covered by exactly one permutation in $X$. Using the above argument, it can be seen that there are exactly $\mu(T)$ rows of \textsf{C} that cover $T$.
\end{proof}

Suppose $X$ is an \textsf{SCA}$(t!;t,v)$. Then every sequence of length $t$ is covered by exactly one permutation in $X$. Therefore, every sequence of length $t - a$ is covered by exactly $t!/a!$ permutations in $X$ for $0 < a < t$. Thus, every array in $\mathcal{C}_{a}(X)$ has $t!/a!$ rows. \lref{t:icax} implies the following theorem.

\begin{theorem}[\cite{Chee13}]
Let $v,t$ and $a$ be integers such that $v \geqslant t \geqslant 3$ and $0 < a < t$. Then \textup{\textsf{SCAN}}$(t,v) \geqslant a!\,\textup{\textsf{CAN}}_{X}(t - a, v - a, a + 1)$. \label{cax}
\end{theorem}

For a row $\rho$ of an array $\textsf{C} \in \mathcal{C}_{a}(X)$, let $m_i$ be the number of times the symbol $i$ appears in $\rho$ for $0 \leq i \leq a$. We call $(m_{0}, \dots, m_{a})$ the \textit{multiplicity vector} of $\rho$.

\begin{lemma}\label{t:mult}
Let $X$ be an \textup{\textsf{SCA}}$(t!;t,v)$, let $0 < a < t$, and let $m_{0}, \dots, m_{a}$ be non-negative integers that sum to $v - a$. Then, across all the arrays in $\mathcal{C}_{a}(X)$, there must be exactly $t!$ rows with multiplicity vector $(m_{0}, \dots, m_{a})$.
\end{lemma}

\begin{proof}
Let $\textsf{C} \in \mathcal{C}_{a}(X)$ and let $A$ be the set of symbols that were deleted from $X$ to generate $\textsf{C}$. In order to generate a row in $\textsf{C}$ with multiplicity vector $(m_{0}, \dots, m_{a})$, the symbols in $A$ must appear in positions $j + \sum_{i = 0}^{j} m_i$ for $0 \leq j \leq a - 1$ in the corresponding permutation in $X$ where positions are indexed from 0. Now, there are $t!$ permutations in $X$ and for each permutation, there is one set of symbols in these positions. Therefore, across all the arrays in $\mathcal{C}_{a}(X)$, there must be exactly $t!$ rows with multiplicity vector $(m_{0}, \dots, m_{a})$.
\end{proof}

Recall that a row of a $\textsf{CA}_{X}(N; t,k,v)$ is \textit{constant} if it contains only one symbol (i.e. $m_{i} = k$ for some $i$). Since the alphabet for each array in $\mathcal{C}_{a}(X)$ has $a+1$ symbols, then by \lref{t:mult}, there are a total of $(a+1)t!$ constant rows across all the arrays in $\mathcal{C}_{a}(X)$. Recall that there are $v!/(v-a)!$ arrays in $\mathcal{C}_{a}(X)$. Therefore, the average number of constant rows per array in $\mathcal{C}_{a}(X)$ is $(a+1)t!(v - a)!/v!$.

\section{Excess coverage arrays with strength 2}\label{ca2}

In this section, we describe computations relating to the existence of excess coverage arrays with strength 2. The results of these computations allow us to prove \tref{unique}. Consider a strength 2 covering array with excess coverage. Let $T = \{ (c_1, \nu_1),$ $(c_2, \nu_2) \}$ be a 2-way interaction. If $\nu_1 = \nu_2$, then $\mu(T) = 2$ but if $\nu_1 \neq \nu_2$, then $\mu(T) = 1$. Following the convention of Chee et al.~\cite{Chee13}, we refer to the former kind of pair as \textit{constant pairs} and the latter as \textit{non-constant pairs}. Let \textsf{C} be a \textsf{CA}$_{X}(N;2,k,v)$ where every 2-way interaction $T$ is covered exactly $\mu(T)$ times. For distinct columns $c_{1}$ and $c_{2}$, there are $v^{2}$ interactions of the form $\{ (c_{1}, \nu_{1}), (c_{2}, \nu_{2})\}$. Of these, $v$ are constant pairs which are covered exactly twice by \textsf{C} while the remaining $v^{2} - v$ interactions are non-constant pairs and are covered exactly once by \textsf{C}. Therefore, \textsf{C} has $N = v^{2} - v + 2v = v(v+1)$ rows. Therefore, \textsf{CAN}$_{X}(2,k,v) \geqslant v(v+1)$. By \lref{t:icax}, a \textsf{CA}$_{X}(v(v+1);2,k,v)$ is necessary for the existence of an \textsf{SCA}$((v+1)!;v+1,v+k-1)$. Chee et al. proved the following upper bound on the number of columns of a \textsf{CA}$_{X}(v(v+1);2,k,v)$.

\begin{theorem}[\cite{Chee13}] \label{caxbound}
For $v \geqslant 4$, if a \textsf{CA}$_{X}(v(v + 1);2,k,v)$ exists, then $k \leq v + 2$. 
\end{theorem}

This theorem, with \tref{cax}, implies that for $t \geq 5$, if an \textsf{SCA}$(t!;t,v)$ exists, then $v \leq 2t - 1$.

We performed a series of computations to find the maximum number of columns $k$ for which a \textsf{CA}$_{X}(v(v+1);2,k,v)$ exists for $v \in \{2,3,4,5,6\}$. The computations for $v=6$ are of particular interest since they relate to the $t=7$ case of Levenshtein's conjecture. This is the smallest value of $t$ for which the existence of \textsf{SCA}s with $t!$ permutations has not been fully resolved. In particular, the results for $v=6$ allow us to prove \tref{t:sca79}. Another consequence of these computations is that the upper bound of $k \leq v + 2$ on the number of columns in a \textsf{CA}$_{X}(v(v+1);2,k,v)$ proved for $v \geq 4$ in \tref{caxbound} also holds for $v \in \{2,3\}$.

For fixed $v \in \{2,3,4,5,6\}$ we recursively found all ways of extending a $\textsf{CA}_{X}(v(v+1); 2,k,v)$ to a $\textsf{CA}_{X}(v(v+1); 2,k+1,v)$. The base of this recursion was the unique $\textsf{CA}_{X}(v(v+1); 2,2,v)$ (this array has two columns and simply contains every non-constant pair once and every constant pair twice). At each step of the recursion, we took a $\textsf{CA}_{X}(v(v+1); 2,k,v)$ and for each symbol found all possible \textit{placements}. A placement for symbol $\nu$ is a set of $v+1$ rows in the $\textsf{CA}_{X}(v(v+1); 2,k,v)$ such that in each column, the symbol $\nu$ appears twice among these rows and all other symbols appear once each. So, if we were to add $\nu$ to a $(k+1)$th column $c$ of these rows, then all 2-way interactions containing the pair $(c, \nu)$ would be covered the suitable number of times. We note that because constant pairs must be covered twice, a placement for symbol $i$ will not be a placement for symbol $j$ for distinct $i$ and $j$. Once all placements were computed, we then found all ways of choosing a placement for each symbol such that these placements partition the rows of the $\textsf{CA}_{X}(v(v+1); 2,k,v)$. Then, to build a \textsf{CA}$_{X}(v(v+1);2,k+1,v)$, we take our original array and add a new column $c$ where the symbol in row $r$ is the symbol whose placement $r$ belongs to.

As part of these computations, it was important to screen for isomorphisms. Two arrays are isomorphic if one can be obtained from the other by permuting the symbols and permuting the columns. It is important to note that any permutation of the symbols must be applied throughout the whole array. Say we were to swap the symbols 0 and 1 in only column 0 of a \textsf{CA}$_{X}(v(v+1);2,k,v)$. Then, the constant pair $\{ (0,0), (1,0) \}$ would now only be covered once while the non-constant pair $\{(0,0), (1,1)\}$ would be covered twice and so the resulting array would not be an excess coverage array. However, if we were to apply the same symbol permutation to every column, we would obtain a \textsf{CA}$_{X}(v(v+1);2,k,v)$. This behaviour contrasts with that of traditional covering arrays in which one can apply different symbol permutations to different columns and still obtain a covering array.

\begin{table}
\begin{center}
\begin{tabular}{c|  c}
$v$ & Maximum number of columns\\
\hline
2 &  4\\
3 &  4\\
4 &  5\\
5 &  6\\
6 &  5\\
\end{tabular}
\caption{Maximum $k$ for which a \textsf{CA}$_{X}(v(v+1);2,k,v)$ exists for $2 \leq v \leq 6$.}
\label{T:maxk}
\end{center}
\end{table}

The maximum number of columns $k$ for which a \textsf{CA}$_{X}(v(v+1);2,k,v)$ exists is recorded in \Tref{T:maxk}. We see that the bound of $k \leqslant v + 2$ on the number of columns of a \textsf{CA}$_{X}(v(v+1);2,k,v)$ also holds for $v \in \{2,3\}$. Indeed, when $v=2$, this bound is met. An example of a \textsf{CA}$_{X}(6;2,4,2)$ is given in Figure~\ref{cax242}. This is the only value of $v$ for which we found a \textsf{CA}$_{X}(v(v+1);2,v+2,v)$. We also see that a \textsf{CA}$_{X}(42;2,6,6)$ does not exist. By \lref{t:icax}, this already shows that an \textsf{SCA}$(7!;7,11)$ does not exist. As stated in \tref{unique}, there is a unique \textsf{CA}$_{X}(42;2,5,6)$. We can use this array and \lref{t:mult} to do better and prove \tref{t:sca79}.

\begin{figure}
\centering
\begin{tabular}{c c c c}
0 & 0 & 0 & 0\\
0 & 0 & 0 & 0\\
0 & 1 & 1 & 1\\
1 & 0 & 1 & 1\\
1 & 1 & 0 & 1\\
1 & 1 & 1 & 0
\end{tabular}
\caption{\label{cax242}A \textsf{CA}$_{X}(6;2,4,2)$.}
\end{figure}

\begin{proof}[Proof of \tref{t:sca79}]
Our computations found exactly one \textsf{CA}$_{X}(42;2,5,6)$ up to isomorphism. This is given in Figure~\ref{fig:design}. We can see that this array contains no constant rows. Our definition of isomorphism preserves the number of constant rows so there are no \textsf{CA}$_{X}(42;2,5,6)$ with any constant rows. However, by \lref{t:mult}, if an \textsf{SCA}$(7!;7,10)$ exists, there must be some \textsf{CA}$_{X}(42;2,5,6)$ with a constant row. Therefore, an \textsf{SCA}$(7!;7,10)$ does not exist. \qedhere

\begin{figure}
\centering
\begin{tabular}{c c c c c | c c c c c | c c c c c}
0 & 0 & 0 & 0 & 1 & 1 & 0 & 2 & 0 & 5 & 2 & 0 & 1 & 1 & 4\\
0 & 0 & 0 & 2 & 0 & 1 & 1 & 1 & 5 & 1 & 2 & 1 & 0 & 3 & 0\\
0 & 1 & 2 & 1 & 2 & 1 & 1 & 3 & 1 & 1 & 2 & 2 & 2 & 2 & 3\\
0 & 2 & 3 & 3 & 4 & 1 & 2 & 0 & 4 & 2 & 2 & 2 & 4 & 0 & 2\\
0 & 3 & 5 & 0 & 0 & 1 & 3 & 4 & 2 & 4 & 2 & 3 & 3 & 5 & 5\\
0 & 4 & 4 & 5 & 3 & 1 & 4 & 1 & 1 & 0 & 2 & 4 & 5 & 4 & 1\\
0 & 5 & 1 & 4 & 5 & 1 & 5 & 5 & 3 & 3 & 2 & 5 & 2 & 2 & 2\\
3 & 0 & 3 & 3 & 3 & 4 & 0 & 5 & 5 & 2 & 5 & 0 & 4 & 4 & 0\\
3 & 1 & 5 & 2 & 5 & 4 & 1 & 4 & 4 & 4 & 5 & 1 & 1 & 0 & 3\\
3 & 2 & 2 & 5 & 0 & 4 & 2 & 1 & 2 & 1 & 5 & 2 & 5 & 1 & 5\\
3 & 3 & 1 & 3 & 2 & 4 & 3 & 0 & 1 & 3 & 5 & 3 & 2 & 3 & 1\\
3 & 3 & 3 & 4 & 3 & 4 & 4 & 2 & 4 & 4 & 5 & 4 & 3 & 2 & 2\\
3 & 4 & 0 & 0 & 4 & 4 & 4 & 4 & 3 & 5 & 5 & 5 & 0 & 5 & 5\\
3 & 5 & 4 & 1 & 1 & 4 & 5 & 3 & 0 & 0 & 5 & 5 & 5 & 5 & 4
\end{tabular}
\caption{\label{fig:design}The unique $\textsf{CA}_{X}(42; 2, 5, 6)$.}
\end{figure}
\end{proof}

Each row of the array in \fref{fig:design} contains a repeated symbol. One can also prove \tref{t:sca79} by arguing that by \lref{t:mult}, if an \textsf{SCA}$(7!;7,10)$ exists, then there must be a \textsf{CA}$_{X}(42;2,5,6)$ with a row with multiplicity vector $(m_{0}, \dots, m_{5})$ such that $m_{i} \leqslant 1$ for $0 \leqslant i \leqslant 5$. However, our computations show that no such array exists.

We remark that since a \textsf{CA}$_{X}(42;2,6,6)$ does not exist, \textsf{CAN}$_{X}(2,6,6) \geqslant 43$. Then, by \tref{cax}, \textsf{SCAN}$(7,11) \geqslant 5160$. Similarly, using \Tref{T:maxk} and \tref{cax}, we find that \textsf{SCAN}$(4,7) \geqslant 26$, \textsf{SCAN}$(5,9) \geqslant 126$ and \textsf{SCAN}$(6,11) \geqslant 744$. Although we have proved that an \textsf{SCA}$(7!;7,10)$ does not exist, the existence of a \textsf{CA}$_{X}(42;2,5,6)$ means we are unable to use \tref{cax} to derive an improved lower bound for \textsf{SCAN}$(7,10)$ as in the other cases above. 

\section{Strength 2 orthogonal arrays}\label{oa2}

In this section, we discuss connections between strength 2 excess coverage arrays and strength 2 orthogonal arrays and analyse the computational results of \sref{ca2} in this context. An \textit{orthogonal array}, denoted \textsf{OA}$(t,k,v)$ is a \textsf{CA}$(N;t,k,v)$ with the property that every $t$-way interaction is covered by exactly 1 row. Unless otherwise specified, when we refer to orthogonal arrays, we are referring to strength 2 orthogonal arrays. We begin by observing upper and lower bounds for \textsf{CAN}$_{X}(2,k,v)$.

\begin{lemma}
For integers $v,k \geqslant 2$,
\begin{equation*}
\textup{\textsf{CAN}}(2,k,v) \leqslant \textup{\textsf{CAN}}_{X}(2,k,v) \leqslant \textup{\textsf{CAN}}(2,k,v) + v.
\end{equation*}
\end{lemma}

\begin{proof}
A \textsf{CA}$_{X}(N;2,k,v)$ is also a \textsf{CA}$(N;2,k,v)$ so the first inequality holds. Let $N = \textsf{CAN}(2,k,v)$ and let $N' = \textsf{CAN}(2,k,v) + v$. Let \textsf{C} be a \textsf{CA}$(N;2,k,v)$ and let \textsf{D} be a $v \times k$ array such that $\textsf{D}[i,j] = i$ for $i \in [v]$ and $j \in [k]$. Then, every 2-way interaction is covered at least once in \textsf{C} and every constant pair is covered exactly once in \textsf{D}. Therefore the rows of \textsf{C} and \textsf{D} form a \textsf{CA}$_{X}(N';2,k,v)$. Thus, the second inequality holds.
\end{proof}

Recall also that \textsf{CAN}$_{X}(2,k,v) \geqslant v(v+1)$. By considering strength 2 orthogonal arrays, we obtain the following corollary.

\begin{cor}
If an \textup{\textsf{OA}}$(2,k,v)$ exists, then \textup{\textsf{CAN}}$_{X}(2,k,v) = v(v+1)$.
\end{cor}

If $v$ is a prime power, then an \textsf{OA}$(2,v+1,v)$ exists~\cite{HSS} and so a \textsf{CA}$_{X}(v(v+1);2,v+1,v)$ also exists. Hence, for infinitely many values of $v$, we can construct a \textsf{CA}$_{X}(v(v+1);2,k,v)$ whose number of columns is one less than the upper bound given by \tref{caxbound}. We continue our discussion of orthogonal arrays by characterising exactly when a \textsf{CA}$_{X}(v(v+1);2,k,v)$ contains an \textsf{OA}$(2,k,v)$.

\begin{lemma}
A \textup{\textsf{CA}}$_{X}(v(v+1); 2,k,v)$ contains a subarray that forms an \textup{\textsf{OA}}$(2,k,v)$ if and only if it contains a constant row for each symbol in $[v]$. \label{const}
\end{lemma}

\begin{proof}
Let $\textsf{C}$ be a \textsf{CA}$_{X}(v(v+1); 2,k,v)$ that contains an \textsf{OA}$(2,k,v)$. In the rows of the \textsf{OA}, every 2-way interaction is covered exactly once so the remaining $v$ rows of $\textsf{C}$ must cover each constant pair once while avoiding all non-constant pairs. The only way to achieve this is if these $v$ rows are all constant and contain distinct symbols of $[v]$. Conversely, if $\textsf{C}$ contains constant rows for each symbol, then the remaining rows must cover every 2-way interaction exactly once and so must form an \textsf{OA}$(2,k,v)$.
\end{proof}

Let $X$ be an \textsf{SCA}$(t!;t,v)$. By \lref{t:mult}, the average number of constant rows per array in $\mathcal{C}_{t-2}(X)$ is $(t - 1)t!(v - t + 2)!/v!$. When $t \geqslant 4$ and $v \geqslant t+2$, this average is less than $t - 1$. An array in $\mathcal{C}_{t-2}(X)$ that contains an \textsf{OA}$(2,v - t + 2, t-1)$ must have at least $t - 1$ constant rows. Therefore, when $t \geqslant 4$ and $v \geqslant t + 2$, there must be some array in $\mathcal{C}_{t - 2}(X)$ that does not contain an orthogonal array. Hence, it is sensible to consider the existence of covering arrays with excess coverage that do not contain orthogonal arrays. \Tref{caxtable} details the number of \textsf{CA}$_{X}(v(v+1);2,k,v)$ that do and do not contain orthogonal arrays for $2 \leq v \leq 6$ and $3 \leq k \leq 6$. \Tref{caxflatrows} records the number of \textsf{CA}$_{X}(42;2,4,6)$ with given numbers of constant rows. If $X$ is an \textsf{SCA}$(7!;7,9)$, then every element of $\mathcal{C}_{5}(X)$ is a \textsf{CA}$_{X}(42;2,4,6)$. Moreover, the average number of constant rows per array in $\mathcal{C}_{5}(X)$ is 2.

\begin{table}
\begin{center}
\begin{tabular}{c || c c c c}
& \multicolumn{4}{c}{$k$}\\
$v$ & 3 & 4 & 5 & 6\\
\hline
2 & 1 $\vert$ 0 & 0 $\vert$ 1 & 0 $\vert$ 0 & 0 $\vert$ 0 \\
3 & 3 $\vert$ 1 & 2 $\vert$ 1 & 0 $\vert$ 0 & 0 $\vert$ 0 \\
4 & 15 $\vert$ 12 & 32 $\vert$ 6 & 80 $\vert$ 5 & 0 $\vert$ 0 \\
5 & 283 $\vert$ 1\,067 & 2\,234 $\vert$ 3\,805 & 104\,146 $\vert$ 348 & 2\,073\,801 $\vert$ 0\\
6 & 190\,472 $\vert$ 1\,666\,259 & 0 $\vert$ 39\,802\,785 & 0 $\vert$ 1 & 0 $\vert$ 0
\end{tabular}
\caption{\label{caxtable}Number of \textsf{CA}$_{X}(v(v+1);2,k,v)$ for different values of $k$ and $v$. Each entry $r \vert s$ indicates that there are $r$ arrays with those parameters that contain an orthogonal array and $s$ arrays that do not contain an orthogonal array.}
\end{center}
\end{table}

\begin{table}
\begin{center}
\begin{tabular}{c c}
Number of constant rows & Number of \textsf{CA}$_{X}(42;2,4,6)$\\
\hline
0 & 3\,034\,497 \\
1 & 10\,418\,576\\
2 & 12\,826\,409\\
3 & 7\,994\,870\\
4 & 2\,393\,309\\
5 & 2\,729\,090\\
6 & 387\,172\\
7 & 18\,385\\
8 & 468\\
9 & 9\\
\end{tabular}
\caption{\label{caxflatrows}The number of CA$_{X}(42; 2,4,6)$ by the number of constant rows.}
\end{center}
\end{table}

After $v=6$, the next three values of $v$ are all prime powers. Hence, for these values of $v$, we will find \textsf{CA}$_{X}(v(v+1);2,k,v)$ containing orthogonal arrays for all $3 \leqslant k \leqslant v + 1$. Recall our definition of isomorphic excess coverage arrays. If we were to permute the symbols of some column in an excess coverage array, in order to obtain another excess coverage array we need to apply the same symbol permutation to all columns. However, if we apply different symbol permutations to different columns of an orthogonal array, we will still obtain an orthogonal array. Hence, we have fewer symmetries to exploit when it comes to excess coverage arrays. Ideally, given that certain sequence covering arrays require the existence of excess coverage arrays that do not contain an orthogonal array, we would like a computational search method that avoids having to generate excess coverage arrays containing orthogonal arrays. If we were to employ the same search methods used for $v \in \{2,3,4,5\}$ in \sref{ca2} for $v > 6$, our catalogues of excess coverage arrays would necessarily include those that contain orthogonal arrays. The relative lack of symmetries of excess coverage arrays would further inflate our catalogue. The following lemma informs one possible strategy for cataloguing excess coverage arrays that do not contain an orthogonal array.


\begin{lemma}
\label{compoa}Let $k \geqslant 5$ and let \textup{\textsf{C}} be a \textup{\textsf{CA}}$_{X}(v(v+1);2,k,v)$ that does not contain an orthogonal array. Then there is some column of \textup{\textsf{C}} such that if that column is removed from \textup{\textsf{C}}, then the resulting \textup{\textsf{CA}}$_{X}(v(v+1);2,k-1,v)$ does not contain an orthogonal array either. 
\end{lemma}

\begin{proof}
As \textsf{C} does not contain an orthogonal array, then by \lref{const} and without loss of generality, there is no constant row in \textsf{C} containing only the symbol 0. Suppose that by deleting any column from \textsf{C}, we obtain a \textsf{CA}$_{X}(v(v+1);2,k-1,v)$ that does contain an orthogonal array. Then, by deleting any column of \textsf{C}, we create a constant row containing the symbol 0. This means that for $0 \leqslant i \leqslant k-1$, there is a row $r$ of \textsf{C} such that $\textsf{C}[r,i] \neq 0$ and $\textsf{C}[r,j] = 0$ for $j \neq i$. However, since $k \geqslant 5$, this means there are at least three rows of \textsf{C} in which the symbols in columns 0 and 1 are both 0. Thus, the 2-way interaction $\{(0,0), (1,0)\}$ is covered at least three times which contradicts the fact that \textsf{C} covers every constant pair twice. Therefore, we must be able to delete some column of \textsf{C} to obtain a \textsf{CA}$_{X}(v(v+1);2,k-1,v)$ that does not contain an orthogonal array.
\end{proof}

The consequence of \lref{compoa} is that a complete catalogue of \textsf{CA}$_{X}(v(v+1);2,4,v)$ that do not contain an orthogonal array can be used to generate a complete catalogue of \textsf{CA}$_{X}(v(v+1);2,5,v)$ that do not contain an orthogonal array which in turn can be used to generate a complete catalogue of \textsf{CA}$_{X}(v(v+1);2,6,v)$ that do not contain an orthogonal array and so on. The key is to generate this initial catalogue. This may be achieved by using a method similar to the computations in \sref{ca2} but instead adding two (or more) placements for each symbol to the unique \textsf{CA}$_{X}(v(v+1);2,2,v)$. For each symbol, each of these placements would correspond to the rows in which that symbol would appear in one of the new columns being added. The placements chosen for each symbol have to intersect in exactly two rows to ensure that each constant pair in the new columns are covered exactly twice. For two distinct symbols, placements corresponding to the same column must be disjoint so that exactly one symbol appears in every entry of the resulting \textsf{CA}$_{X}$. The placement corresponding to one symbol in one column and the placement corresponding to the other symbol in a different column must intersect in one row to ensure every non-constant pair in the new columns is covered exactly once.

Furthermore, to ensure that the arrays we build do not contain orthogonal arrays, we can without loss of generality ensure that these arrays do not contain a constant row of 0s. This can be done by choosing the placements for the symbol 0 such that for each of the two constant rows containing 0s in the \textsf{CA}$_{X}(v(v+1);2,2,v)$, there is a placement avoiding that row.

The bound $k \geqslant 5$ in \lref{compoa} is best possible in the sense that there does exist a \textsf{CA}$_{X}(v(v+1);2,4,v)$ that does not contain an orthogonal array such that removing any column from this array gives a \textsf{CA}$_{X}(v(v+1);2,3,v)$ that does contain an orthogonal array. An example of such an array is the \textsf{CA}$_{X}(6;2,4,2)$ in \fref{cax242}. This array contains a constant row of 0s but no constant row of 1s. However, deleting any column will form a constant row of 1s. Another example is the \textsf{CA}$_{X}(42; 2,4,6)$ in \fref{cax644}. In this array, there is a constant row for all symbols except 5. The rows 0555, 5155, 5505 and 5552 ensure that the deletion of any column creates a constant row of 5s. Thus, it is necessary that we start with a catalogue of arrays with four columns.

\begin{figure}
\centering
\begin{tabular}{c c c c | c c c c | c c c c}
0 & 0 & 0 & 0 & 1 & 0 & 2 & 2 & 2 & 0 & 4 & 5\\
0 & 0 & 1 & 1 & 1 & 1 & 1 & 1 & 2 & 1 & 2 & 1\\
0 & 1 & 0 & 2 & 1 & 1 & 3 & 3 & 2 & 2 & 2 & 2\\
0 & 2 & 2 & 0 & 1 & 2 & 0 & 1 & 2 & 2 & 5 & 4\\
0 & 3 & 3 & 4 & 1 & 3 & 5 & 0 & 2 & 3 & 0 & 3\\
0 & 4 & 4 & 3 & 1 & 4 & 1 & 5 & 2 & 4 & 3 & 2\\
0 & 5 & 5 & 5 & 1 & 5 & 4 & 4 & 2 & 5 & 1 & 0\\
3 & 0 & 5 & 3 & 4 & 0 & 0 & 4 & 5 & 0 & 3 & 0\\
3 & 1 & 1 & 4 & 4 & 1 & 4 & 0 & 5 & 1 & 5 & 5\\
3 & 2 & 4 & 2 & 4 & 2 & 3 & 5 & 5 & 2 & 1 & 3\\
3 & 3 & 2 & 5 & 4 & 3 & 1 & 2 & 5 & 3 & 4 & 1\\
3 & 3 & 3 & 3 & 4 & 4 & 4 & 4 & 5 & 4 & 2 & 4\\
3 & 4 & 0 & 0 & 4 & 4 & 5 & 1 & 5 & 5 & 0 & 5\\
3 & 5 & 3 & 1 & 4 & 5 & 2 & 3 & 5 & 5 & 5 & 2
\end{tabular}
\caption{\label{cax644} A \textsf{CA}$_{X}(42;2,4,6)$. Although it does not contain an orthogonal array, the deletion of any column forms a \textsf{CA}$_{X}(42;2,3,6)$ that does contain an orthogonal array.}
\end{figure}

We conclude this section by returning to the unique \textsf{CA}$_{X}(42;2,5,6)$ in \fref{fig:design}. An \textsf{OA}$(2,5,6)$ does not exist so this excess coverage array cannot contain an orthogonal array. This is confirmed by \lref{const} and the fact that the excess coverage array has no constant rows. We found that this excess coverage array has an automorphism group isomorphic to the alternating group $A_{4}$. This group has two orbits on the columns of the array: one contains just the first column, and the other contains the remaining columns. Up to isomorphism, there are two possible arrays we can obtain by deleting a column of this $\textsf{CA}_{X}(42; 2,5,6)$: we can either delete the column that is in its own orbit or we can delete one of the other four. As there is no \textsf{OA}$(2,4,6)$, neither of these arrays contains an orthogonal array. Similarly, there are two possible arrays we can obtain by deleting two columns. Either we delete two of the columns in the orbit of size 4 or we delete one of these columns and the column in its own orbit. The second of these arrays contains an \textsf{OA}$(2,3,6)$. This orthogonal array can be constructed by deleting the first 2 columns from the array in \fref{fig:design} and then for each symbol, removing a constant row that is necessarily formed from this column deletion. By treating the rows of this orthogonal array as triples containing a row, column and symbol, we can form a $6\times 6$ Latin square. This Latin square, shown in \fref{latin}, has the highest number of transversals of any $6\times 6$ Latin square.

\begin{figure}
\begin{center}
\begin{tabular}{c c c c c c}
0 & 2 & 3 & 5 & 4 & 1\\
2 & 3 & 0 & 4 & 1 & 5\\
1 & 0 & 4 & 3 & 5 & 2\\
3 & 5 & 2 & 1 & 0 & 4\\
5 & 4 & 1 & 0 & 2 & 3\\
4 & 1 & 5 & 2 & 3 & 0
\end{tabular}
\caption{\label{latin} Latin square obtained by deleting 2 columns from the array in Figure~\ref{fig:design}}
\end{center}
\end{figure}


\section{Binary covering arrays}\label{binary}

In the previous sections, we have been concerned with excess coverage arrays with strength 2. We can build such arrays using \conref{construction} with $a = t-2$. However, \conref{construction} can be applied for any $1 \leqslant a \leqslant t-1$ so we are not bound to looking solely at strength 2 excess coverage arrays. In this section, we consider excess coverage arrays with $v=2$. That is, \textsf{CA}$_{X}(N;t,k,2)$. These arrays may be built by letting $a=1$ in \conref{construction}. In particular, we study these arrays to see whether they may provide additional insight into Levenshtein's conjecture and the existence of an \textsf{SCA}$(t!;t,v)$. We consider rows of these arrays as being elements of $\{0,1\}^{k}$. For $u \in \{0,1\}^{k}$, the \textit{weight} of $u$ is the number of 1s in $u$.

\begin{lemma}
If \textup{\textsf{C}} is a \textup{\textsf{CA}}$_{X}(N;t,k,2)$ such that \textup{\textsf{C}} covers each $t$-way interaction $T$ exactly $\mu(T)$ times, then $N = (t+1)!$.
\end{lemma}

\begin{proof}
Let $T = \{(c_{i}, \nu_{i}) : 0 \leq i \leq t-1\}$ be a $t$-way interaction with $\vert \tau_{0}(T) \vert = j$. Then $\mu(T) = j ! (t - j)!$. For $t$ distinct columns, there are $\binom{t}{j}$ interactions on those columns with $\vert \tau_{0}(T) \vert = j$. Therefore,
\begin{equation*}
N = \sum_{j=0}^{t} \binom{t}{j} j!(t-j)! = \sum_{j=0}^{t} t! = (t+1)!. \qedhere
\end{equation*} 
\end{proof}

We next analyse and count the number of \textsf{CA}$_{X}((t+1)!;t,t+1,2)$.

\begin{lemma}
Let $\textup{\textsf{C}}$ be a \textup{\textsf{CA}}$_{X}((t+1)!;t,t+1,2)$. For $0 \leqslant i \leqslant t + 1$, there is a constant $x_{i}$ such that each element in $\{0,1\}^{t+1}$ with weight $i$ appears as a row of $\textup{\textsf{C}}$ exactly $x_{i}$ times. Furthermore,
\begin{equation} 
x_{i} + x_{i + 1} = i!(t - i)!, \text{ for } 0 \leq i \leq t. \label{e:eq1}
\end{equation}
\end{lemma}

\begin{proof}
We proceed by induction on $i$. As there is a single element of $\{0,1\}^{t+1}$ with weight 0, we can take $x_{0}$ to be the number of weight 0 rows in \textsf{C}. Let $x_{i}$ be the number of times each element of $\{0,1\}^{t+1}$ with weight $i$ appears as a row of $\textsf{C}$. Let $u \in \{0,1\}^{t + 1}$ have weight $i+1$ and let $T$ be a $t$-way interaction with $\vert \tau_{1}(T) \vert = i$ such that $u$ covers $T$. Let $w$ be the unique vector in $\{0,1\}^{t+1}$ with weight $i$ covering $T$. Then $u$ and $w$ are the only vectors in $\{0,1\}^{t+1}$ that cover $T$. As $\vert \tau_{1}(T) \vert = i$, we must have $\mu(T) = i!(t - i)!$. As there are $x_{i}$ rows in $\textsf{C}$ equal to $w$, there must be $i!(t - i)! - x_{i}$ rows equal to $u$. Therefore, each element of $\{0,1\}^{t+1}$ with weight $i+1$ appears as a row of $\textsf{C}$ $x_{i+1}$ times where $x_{i+1} = i!(t - i)! - x_{i}$.
\end{proof}

\begin{theorem}
The number of distinct \textup{\textsf{CA}}$_{X}((t+1)!; t, t + 1, 2)$ is exactly $\lfloor \frac{t}{2} \rfloor ! \lceil \frac{t}{2} \rceil ! + 1$ for $t \geq 1$. When $t$ is odd, no two of these arrays are isomorphic but when $t$ is even, there are $\lfloor \frac{t}{2} \rfloor ! \lceil \frac{t}{2} \rceil ! / 2$ isomorphism classes each containing two arrays and one isomorphism class containing a single array.
\end{theorem}

\begin{proof}
From each non-negative integer solution $x_{0}, \dots, x_{t+1}$ to the system in \eqref{e:eq1}, we can build a \textsf{CA}$_{X}((t+1)!; t, t+1, 2)$ by adding $x_{i}$ rows corresponding to each vector in $\{0,1\}^{t+1}$ with weight $i$. Therefore, the number of distinct \textsf{CA}$_{X}((t+1)!;t,t+1,2)$ is equal to the number of non-negative integer solutions to \eqref{e:eq1}.

Let $s = \lfloor t/2 \rfloor$. Now consider $x_{s}$. By \eqref{e:eq1}, $x_{s} + x_{s + 1} = \lfloor \frac{t}{2} \rfloor ! \lceil \frac{t}{2} \rceil !$. Let $x_{s}$ be some integer between 0 and $\lfloor \frac{t}{2} \rfloor ! \lceil \frac{t}{2} \rceil !$. Then, in order to satisfy each equation in \eqref{e:eq1}, fixing $x_{s}$ fixes the values for all other $x_{i}$. Suppose $0 \leq i \leq s-1$. Then $i!(t - i)! > (i + 1)!(t - (i + 1))!$. Therefore, if $x_{i + 1}$ is a non-negative integer that is at most $(i + 1)!(t - (i + 1))!$, then $x_{i}$ is a non-negative integer that is at most $i!(t - i)!$. As $x_{s}$ is a non-negative integer that is at most $\lfloor \frac{t}{2} \rfloor ! \lceil \frac{t}{2} \rceil !$, then $x_{i}$ is a non-negative integer for all $i \in \{0,\dots, s\}$. Similarly, $x_{s + 1}$ is a non-negative integer that is at most $\lfloor \frac{t}{2} \rfloor ! \lceil \frac{t}{2} \rceil !$ and $i!(t - i)! < (i + 1)!(t - (i + 1))!$ for $s+1 \leqslant i \leqslant t - 1$, so $x_{i}$ is a non-negative integer for all $i \in \{s+1,\dots,t+1\}$. Thus, each non-negative integer choice for $x_{s}$ gives rise to a non-negative integer solution to \eqref{e:eq1} and thus, a \textsf{CA}$_{X}((t+1)!;t,t + 1, 2)$. Therefore, the number of \textsf{CA}$_{X}((t+1)!;t,t + 1, 2)$ is exactly $\lfloor \frac{t}{2} \rfloor ! \lceil \frac{t}{2} \rceil ! + 1$.

Now we can consider isomorphisms of these arrays. Let \textsf{C} be a \textsf{CA}$_{X}((t+1)!;t,t+1,2)$. Since each vector in $\{0,1\}^{t+1}$ with the same weight appears the same number of times in $\textsf{C}$, permuting the columns of $\textsf{C}$ will not generate a new array. The only other possible isomorphism to consider then is swapping all 1s and 0s in $\textsf{C}$. By \eqref{e:eq1}, 
\begin{equation}\label{e:eq3}
x_{i} + x_{i + 1} = x_{t+1-i} + x_{t-i}.
\end{equation}
Suppose $t$ is odd. Then, $x_{(t+1)/2 - 1} + x_{(t+1)/2} = x_{(t+1)/2} + x_{(t+1)/2 + 1}$. Thus, $x_{(t+1)/2 - 1} = x_{(t+1)/2 + 1}$ and so $x_{i} = x_{t + 1 - i}$ for $0 \leq i \leq (t+1)/2$. Hence, swapping 1s and 0s in a \textsf{CA}$_{X}((t+1)!; t,t + 1, 2)$ when $t$ is odd will not generate a new array. Therefore, the $\lfloor \frac{t}{2} \rfloor ! \lceil \frac{t}{2} \rceil ! + 1$ possible arrays with those parameters must be in distinct isomorphism classes.

Now suppose $t$ is even. Then it can be shown by induction on $i$ and using \eqref{e:eq3} that $x_{t/2 - i} = x_{t/2 + 1 + i} + (-1)^{i}(x_{t/2} - x_{t/2 + 1})$ for $0 \leq i \leq t/2$. In particular, $x_{i} = x_{t+1 - i}$ for $0 \leq i \leq t+1$ if and only if $x_{t/2} = x_{t/2 + 1}$. Thus, if $x_{t/2} \neq x_{t/2+1}$ swapping 1s and 0s will generate a new array. Therefore, when $t$ is even, there are $\lfloor \frac{t}{2} \rfloor ! \lceil \frac{t}{2} \rceil ! / 2$ isomorphism classes of \textsf{CA}$_{X}((t+1)!; t,t+1,2)$ each containing 2 arrays and 1 isomorphism class containing a single array.
\end{proof}

To conclude this section, we consider whether analysing the existence of \textsf{CA}$_{X}((t+1)!; t,k,2)$ may lead to improvements to the upper bound of Chee et al.~\cite{Chee13} that the number of symbols in an \textsf{SCA}$(t!;t,v)$ is at most $2t - 1$. Suppose an \textsf{SCA}$(t!;t,2t-1)$ exists. Then from this array, we can obtain a \textsf{CA}$_{X}(t!; t-1,2t-2,2)$. That is, we obtain an excess coverage array with 2 symbols and where the number of columns is twice the strength. Moreover, this array will cover every interaction $T$ exactly $\mu(T)$ times. Therefore, if we can establish the non-existence of a \textsf{CA}$_{X}((t+1)!; t,2t,2)$, for some value of $t$, then we can improve upon the bound of Chee et al.

In building a \textsf{CA}$_{X}((t+1)!; t,2t,2)$ we essentially need to decide how many times each vector in $\{0,1\}^{2t}$ will appear in the array. For a $t$-way interaction $T$, there are $2^{t}$ vectors in $\{0,1\}^{2t}$ that cover $T$. The number of times these vectors collectively appear in a \textsf{CA}$_{X}((t+1)!; t,2t,2)$ must be $\mu(T)$. If we apply this logic to all possible $t$-way interactions, we can obtain a system of $\binom{2t}{t}2^{t}$ equations with $2^{2t}$ variables. A non-negative integer solution to this system can be used to generate a \textsf{CA}$_{X}((t+1)!; t,2t,2)$ by adding the appropriate number of each vector as specified by the solution.

We can simplify this system by only considering arrays in which vectors with the same weight appear the same number of times. Note that this restriction is not necessary as it was in the case where $k = t+1$. By making this restriction, we reduce the number of variables to $2t + 1$, one for each possible weight, and the number of equations to $t+1$ since all interactions with the same value of $\vert \tau_{1} \vert$ will generate the same equation. Specifically, if $\vert \tau_{1}(T) \vert = i$ for some interaction $T$, then there are $\binom{t}{j}$ vectors with weight $i+j$ that cover $T$. Let $x_{i}$ be the number of times each vector with weight $i$ appears in the array. Then, the relevant system of equations is
\begin{equation}
\sum_{j=0}^{t} \binom{t}{j} x_{i+j} = i!(t-i)!, \text{ for } 0 \leq i \leq t. \label{eq:2}
\end{equation}

We have found non-negative integer solutions to \eqref{eq:2} for all $t \leq 17$. Therefore, a $\textsf{CA}_{X}((t+1)!;t,2t,2)$ exists for all $t \leqslant 17$. These arrays mean that at present, we cannot disprove the existence of an \textsf{SCA}$(t!;t,2t-1)$ for $t \leqslant 18$. We found that no non-negative integer solution to \eqref{eq:2} exists for $t = 18$. However, this does not necessarily rule out the existence of a \textsf{CA}$_{X}(19!;18,36,2)$. Hence, we are also at present unable to disprove the existence of an \textsf{SCA}$(19!;19,37)$.

\section{Conclusion}\label{concl}

In closing, we highlight some open problems in the area. Perhaps the most obvious is the resolution of Levenshtein's conjecture for $t = 7$. By \tref{t:sca79}, $v = 9$ is the only value of $v$ for which the existence of an \textsf{SCA}$(7!;7,v)$ is unresolved.

Less is known about the conjecture for $t \geqslant 8$. In an attempt to address the $t = 8$ case, we briefly tried computing \textsf{CA}$_{X}(56;2,k,7)$ however we found the large number of possible placements in the \textsf{CA}$_{X}(56;2,2,7)$ a significant barrier to these efforts. Moreover, given the sharp increase in the number of \textsf{CA}$_{X}(v(v+1);2,3,v)$ and \textsf{CA}$_{X}(v(v+1);2,4,v)$ as $v$ increases shown in \Tref{caxtable}, we suspect the number of \textsf{CA}$_{X}(v(v+1);2,3,7)$ and \textsf{CA}$_{X}(v(v+1);2,4,7)$ may be so large as to render our method of generating exhaustive catalogues difficult.

We believe excess coverage arrays are interesting objects in their own right. One potential problem may be to consider whether a \textsf{CA}$_{X}(v(v+1);2,v+2,v)$ exists for any value of $v$ other than $v = 2$.

\section*{Acknowledgements}
The first author was supported by an Australian Government Research and Training Program Scholarship and a Monash University Postgraduate Publication Award.


\end{document}